\documentclass{amsart}

\usepackage{amsfonts,amscd}
\usepackage{amssymb}
\usepackage{url}

\theoremstyle{plain}
\newtheorem{theorem}                {Theorem}      [section]
\newtheorem{proposition}  [theorem]  {Proposition}

\theoremstyle{definition}

\newtheorem{remark}       [theorem]  {Remark}

\DeclareMathOperator{\trace}{trace}
\DeclareMathOperator{\grad}{grad}

\numberwithin{equation}{section}

\begin{document}

\title{ON THE GEOMETRY OF BIHARMONIC SUBMANIFOLDS IN SASAKIAN SPACE
FORMS}

\thanks{Contribution to the Proceedings of the $10$-th International
Conference on Geometry, Integrability and Quantization, Varna 2008,
Bulgaria}
\author{D.~Fetcu}
\author{C.~Oniciuc}

\address{Department of Mathematics\\
"Gh. Asachi" Technical University of Iasi\\
Bd. Carol I no. 11 \\
700506 Iasi, Romania} \email{dfetcu@math.tuiasi.ro}

\address{Faculty of Mathematics\\ ``Al.I. Cuza'' University of Iasi\\
Bd. Carol I no. 11 \\
700506 Iasi, Romania} \email{oniciucc@uaic.ro}

\begin{abstract}
We classify all proper-biharmonic Legendre curves in a Sasakian
space form and point out some of their geometric properties. Then
we provide a method for constructing anti-invariant
proper-biharmonic submanifolds in Sasakian space forms. Finally,
using the Boothby-Wang fibration, we determine all
proper-biharmonic Hopf cylinders over homogeneous real
hypersurfaces in complex projective spaces.
\end{abstract}

\maketitle

\section{Introduction}\label{sec:1}

As defined by Eells and Sampson in \cite{Eells}, {\em harmonic
maps} $f:(M,g)\rightarrow(N,h)$ are the critical points of the
{\em energy functional}
$$
E(f)=\frac{1}{2}\int_{M}\| df\|^{2} \ v_{g}
$$
and they are solutions of the associated Euler-Lagrange equation
$$
\tau(f)={\mathrm{trace}_{g}}{\nabla df}=0,
$$
where $\tau(f)$ is called the {\em tension field} of $f$. When $f$
is an isometric immersion with mean curvature vector field $H$,
then $\tau(f)=mH$ and $f$ is harmonic if and only if it is
minimal.

The {\em bienergy functional} (proposed also by Eells and Sampson
in 1964, \cite{Eells}) is defined by
$$
E_{2}(f)=\frac{1}{2}\int_{M}\|\tau(f)\|^{2} \ v_{g}.
$$
The critical points of $E_{2}$  are called {\em biharmonic maps}
and they are solutions of the Euler-Lagrange equation (derived by
Jiang in 1986, \cite{Jiang}):
$$
\tau_{2}(f)=-\Delta^{f}\tau(f)-\trace_{g}R^{N}(df,\tau(f))df=0,
$$
where $\Delta^{f}$ is the Laplacian on sections of $f^{-1}TN$ and
$R^{N}(X,Y)=\nabla_X\nabla_Y-\nabla_Y\nabla_X-\nabla_{[X,Y]}$ is
the curvature operator on $N$; $\tau_{2}(f)$ is called the {\em
bitension field} of $f$. Since all harmonic maps are biharmonic,
we are interested in studying those which are biharmonic but
non-harmonic, called {\em proper-biharmonic} maps.

Now, if $f:M\rightarrow N_c$ is an isometric immersion into a
space form of constant sectional curvature $c$, then
$$
\tau(f)=mH \quad \textnormal{and} \quad \tau_2(f)=-m \Delta^{f}H+c
m^2 H.
$$
Thus $f$ is biharmonic if and only if
$$
\Delta^{f} H=mcH.
$$

In a different way, Chen defined the biharmonic submanifolds in an
Euclidean space as those with harmonic mean curvature vector field
(\cite{Chen}). Replacing $c=0$ in the above equation we just
reobtain Chen's definition. Moreover, let
$f:M\rightarrow\mathbb{R}^n$ be an isometric immersion. Set
$f=(f^1,\ldots,f^n)$ and $H=(H^1,\ldots,H^n)$. Then
$\Delta^{f}H=(\Delta H^1,\ldots,\Delta H^n)$, where $\Delta$ is
the Beltrami-Laplace operator on $M$, and $f$ is biharmonic if and
only if
$$
\Delta^{f}H=\Delta(\frac{-\Delta f}{m})=-\frac{1}{m}\Delta^2 f=0.
$$

There are several classification results for the proper-biharmonic
submanifolds in Euclidean spheres and non-existence results for
such submanifolds in space forms $N_c$, $c\leq 0$ (\cite{BMO},
\cite{BMO2}, \cite{RCSMCO}, \cite{RCSMCO2}, \cite{RCSMPP},
\cite{Chen}, \cite{ID}), while in spaces of non-constant sectional
curvature only few results were obtained (\cite{Sasahara2},
\cite{CIL}, \cite{Ura}, \cite{Ino}, \cite{Sasahara1},
\cite{Zhang}).

We recall that the proper-biharmonic curves of the unit Euclidean
$2$-dimensional sphere $\mathbb{S}^2$ are the circles of radius
$\frac{1}{\sqrt{2}}$, and the proper-biharmonic curves of
$\mathbb{S}^3$ are the geodesics of the minimal Clifford torus
$\mathbb{S}^1(\frac{1}{\sqrt{2}})\times\mathbb{S}^1(\frac{1}{\sqrt{2}})$
with the slope different from $\pm 1$. The proper-biharmonic
curves of $\mathbb{S}^3$ are helices. Further, the
proper-biharmonic curves of $\mathbb{S}^n$, $n>3$, are those of
$\mathbb{S}^3$ (up to a totally geodesic embedding). Concerning
the hypersurfaces of $\mathbb{S}^n$, it was conjectured in
~\cite{BMO} that the only proper-biharmonic hypersurfaces are the
open parts of $\mathbb{S}^{n-1}(\frac{1}{\sqrt{2}})$ or
$\mathbb{S}^{m_1}(\frac{1}{\sqrt{2}})\times
\mathbb{S}^{m_2}(\frac{1}{\sqrt{2}})$ with $m_1+m_2=n-1$ and
$m_1\neq m_2$.

Since odd dimensional unit Euclidean spheres $\mathbb{S}^{2n+1}$
are Sasakian space forms with constant $\varphi$-sectional
curvature $1$, the next step is to study the biharmonic
submanifolds of Sasakian space forms. In this paper we mainly
gather the results obtained in \cite{FetcuOniciuc}, \cite{DFCO1}
and \cite{DFCO2}.

We note that the proper-biharmonic submanifolds in
pseudo-Riemannian manifolds are also intensively-studied (for
example, see \cite{Arv}, \cite{Arv2}, \cite{Chen1}).

For a general account of biharmonic maps see
\cite{MontaldoOniciuc} and \textit{The Bibliography of Biharmonic
Maps} \cite{bibl}.

\noindent \textbf{Conventions.} We work in the $C^{\infty}$
category, that means manifolds, metrics, connections and maps are
smooth. The Lie algebra of the vector fields on $N$ is denoted by
$C(TN)$.

\section{Sasakian Space Forms}\label{sec:2}

In this section we briefly recall some basic facts from the theory
of Sasakian manifolds. For more details see \cite{Blair}.

A \textit{contact metric structure} on a manifold $N^{2n+1}$ is
given by $(\varphi,\xi,\eta,g)$, where $\varphi$ is a tensor field
of type $(1,1)$ on $N$, $\xi$ is a vector field on $N$, $\eta$ is
an 1-form on $N$ and $g$ is a Riemannian metric, such that
$$
\left\{\begin{array}{cc} \varphi^{2}=-I+\eta\otimes\xi, \quad
\eta(\xi)=1,\\ \\g(\varphi X,\varphi Y)=g(X,Y)-\eta(X)\eta(Y),
\quad g(X,\varphi Y)=d\eta(X,Y),
\end{array}\right.
$$
for any $X,Y\in C(TN)$.

A contact metric structure $(\varphi,\xi,\eta,g)$ is
\textit{Sasakian} if it is \textit{normal}, i.e.
$$
N_{\varphi}+2d\eta\otimes\xi=0,
$$
where
\newpage
$$
N_{\varphi}(X,Y)=[\varphi X,\varphi Y]-\varphi \lbrack \varphi
X,Y]-\varphi \lbrack X,\varphi Y]+\varphi ^{2}[X,Y], \quad \forall
X,Y\in C(TN)
$$
is the Nijenhuis tensor field of $\varphi$.

The \textit{contact distribution} of a Sasakian manifold
$(N,\varphi,\xi,\eta,g)$ is defined by $\{X\in TN : \eta(X)=0\}$,
and an integral curve of the contact distribution is called
\textit{Legendre curve}.

A submanifold $M$ of $N$ which is tangent to $\xi$ is said to be
\textit{anti-invariant} if $\varphi$ maps any vector tangent to
$M$ and normal to $\xi$ to a vector normal to $M$.

Let $(N,\varphi,\xi,\eta,g)$ be a Sasakian manifold. The sectional
curvature of a 2-plane generated by $X$ and $\varphi X$, where $X$
is an unit vector orthogonal to $\xi$, is called
$\varphi$-\textit{sectional curvature} determined by $X$. A
Sasakian manifold with constant $\varphi$-sectional curvature $c$
is called a \textit{Sasakian space form} and it is denoted by
$N(c)$.

A contact metric manifold $(N,\varphi,\xi,\eta,g)$ is called
\textit{regular} if for any point $p\in N$ there exists a cubic
neighborhood of $p$ such that any integral curve of $\xi$ passes
through the neighborhood at most once, and \textit{strictly
regular} if all integral curves are homeomorphic to each other.

\noindent Let $(N,\varphi,\xi,\eta,g)$ be a regular contact metric
manifold. Then the orbit space $\bar{N}=N/\xi$ has a natural
manifold structure and, moreover, if $N$ is compact then $N$ is a
principal circle bundle over $\bar{N}$ (the Boothby-Wang Theorem).
In this case the fibration $\pi:N\rightarrow\bar{N}$ is called
\textit{the Boothby-Wang fibration}. The Hopf fibration
$\pi:\mathbb{S}^{2n+1}\rightarrow \mathbb{C}P^{n}$ is a well-known
example of a Boothby-Wang fibration.

\begin{theorem}[\cite{Ogiue}] Let $(N,\varphi,\xi,\eta,g)$ be a strictly
regular Sasakian manifold. Then on $\bar{N}$ can be given the
structure of a K\"{a}hler manifold. Moreover, if $(N,\varphi,\xi,$
$\eta,g)$ is a Sasakian space form $N(c)$, then $\bar{N}$ has
constant sectional holomorphic curvature $c+3$.
\end{theorem}

Even if $N$ is non-compact, we still call the fibration
$\pi:N\to\bar{N}$ of a strictly regular Sasakian manifold, the
Boothby-Wang fibration.

\section{Biharmonic Legendre Curves in Sasakian Space Forms}

Let $(N^{n},g)$ be a Riemannian manifold and $\gamma:I\to N$ a
curve parametrized by arc length. Then $\gamma$ is called a
\textit{Frenet curve of osculating order r}, $1\leq r\leq n$, if
there exists orthonormal vector fields $E_{1},E_{2},\ldots,E_{r}$
along $\gamma$ such that $E_{1}=\gamma'=T$,
$\nabla_{T}E_{1}=\kappa_{1}E_{2}$,
$\nabla_{T}E_{2}=-\kappa_{1}E_{1}+\kappa_{2}E_{3}$,...,$\nabla_{T}E_{r}=-\kappa_{r-1}E_{r-1}$,
where $\kappa_{1},\ldots,\kappa_{r-1}$ are positive functions on
$I$.

A geodesic is a Frenet curve of osculating order 1; a
\textit{circle} is a Frenet curve of osculating order 2 with
$\kappa_{1}=\textnormal{constant}$; a \textit{helix of order r},
$r\geq 3$, is a Frenet curve of osculating order $r$ with
$\kappa_{1},\ldots,\kappa_{r-1}$ constants; a helix of order $3$
is called, simply, helix.

In \cite{DFCO1} we studied the biharmonicity of Legendre Frenet
curves and we obtained the following results.

Let $(N^{2n+1},\varphi,\xi,\eta,g)$ be a Sasakian space form with
constant $\varphi$-sectional curvature $c$ and
$\gamma:I\rightarrow N$ a Legendre Frenet curve of osculating
order $r$. Then $\gamma$ is biharmonic if and only if
$$
\begin{array}{rl}
\tau_{2}(\gamma)=&\nabla_{T}^{3}T-R(T,\nabla_{T}T)T\\ \\
=&(-3\kappa_{1}\kappa_{1}')E_{1}+\Big(\kappa_{1}''-\kappa_{1}^{3}-
\kappa_{1}\kappa_{2}^{2}+\frac{(c+3)\kappa_{1}}{4}\Big)E_{2}\\
\\&+(2\kappa_{1}'\kappa_{2}+\kappa_{1}\kappa_{2}')E_{3}+
\kappa_{1}\kappa_{2}\kappa_{3}E_{4}+
\frac{3(c-1)\kappa_{1}}{4}g(E_{2},\varphi T)\varphi T\\ \\=&0.
\end{array}
$$

The expression of the bitension field $\tau_{2}(\gamma)$ imposed a
case-by-case analysis as follows.

\noindent\textbf{Case I ($c=1$)}
\begin{theorem}[\cite{DFCO1}]\label{teocase1}If $c=1$ then $\gamma$ is
proper-biharmonic if and only if $n\geq 2$ and either $\gamma$ is
a circle with $\kappa_{1}=1$ or $\gamma$ is a helix with
$\kappa_{1}^{2}+\kappa_{2}^{2}=1$.
\end{theorem}

\noindent\textbf{Case II ($c\neq 1$ and $E_{2}\perp\varphi T$)}

\begin{theorem}[\cite{DFCO1}]\label{teocase2}Assume that $c\neq 1$ and
$E_{2}\perp\varphi T$. We have
\begin{enumerate}
\item[1)] if $c\leq -3$ then $\gamma$ is biharmonic if and only if
it is a geodesic; \item[2)] if $c>-3$ then $\gamma$ is
proper-biharmonic if and only if either
\begin{enumerate}
\item[a)] $n\geq 2$ and $\gamma$ is a circle with
$\kappa_{1}^{2}=\frac{c+3}{4}$, or \item[b)] $n\geq 3$ and
$\gamma$ is a helix with
$\kappa_{1}^{2}+\kappa_{2}^{2}=\frac{c+3}{4}$.
\end{enumerate}
\end{enumerate}
\end{theorem}

\noindent\textbf{Case III ($c\neq 1$ and $E_{2}\parallel\varphi
T$)}

\begin{theorem}[\cite{DFCO1}]\label{teocase3}If $c\neq 1$ and
$E_{2}\parallel\varphi T$, then $\{T,\varphi T,\xi\}$ is the
Frenet frame field of $\gamma$ and we have
\begin{enumerate}
\item[1)] if $c<1$ then $\gamma$ is biharmonic if and only if it
is a geodesic; \item[2)] if $c>1$ then $\gamma$ is
proper-biharmonic if and only if it is a helix with
$\kappa_{1}^{2}=c-1$ (and $\kappa_{2}=1$).
\end{enumerate}
\end{theorem}

\begin{remark} In dimension 3 the result was obtained by Inoguchi
in \cite{Ino} and explicit examples are given in
\cite{FetcuOniciuc}.
\end{remark}

\noindent\textbf{Case IV ($c\neq 1$ and $g(E_{2},\varphi T)$ is
not constant $0$, $1$ or $-1$)}

\begin{theorem}[\cite{DFCO1}]\label{teocase4}Let $c\neq 1$
and $\gamma$ a Legendre Frenet curve of osculating order $r$ such
that $g(E_{2},\varphi T)$ is not constant $0$, $1$ or $-1$. We
have
\begin{enumerate}
\item[1)] if $c\leq -3$ then $\gamma$ is biharmonic if and only if
it is a geodesic; \item[2)] if $c>-3$ then $\gamma$ is
proper-biharmonic if and only if $r\geq 4$, $\varphi
T=\cos\alpha_{0}E_{2}+\sin\alpha_{0}E_{4}$ and
$$
\left\{\begin{array}{ll}
\kappa_{1},\kappa_{2},\kappa_{3}=\textnormal{constant}>0
\\ \\
\kappa_{1}^{2}+\kappa_{2}^{2}=\frac{c+3}{4}+\frac{3(c-1)}{4}\cos^{2}\alpha_{0}
\\ \\
\kappa_{2}\kappa_{3}=-\frac{3(c-1)}{8}\sin(2\alpha_{0}),
\end{array}\right.
$$
where $\alpha_{0}\in
(0,2\pi)\setminus\{\frac{\pi}{2},\pi,\frac{3\pi}{2}\}$ is a
constant such that
$$
c+3+3(c-1)\cos^{2}\alpha_{0}>0, \quad 3(c-1)\sin(2\alpha_{0})<0.
$$
\end{enumerate}
\end{theorem}

In order to obtain explicit examples of proper-biharmonic Legendre
curves given by Theorem \ref{teocase1} we used the unit Euclidean
sphere $\mathbb{S}^{2n+1}$ as a model of a Sasakian space form
with $c=1$ and we proved the following

\begin{theorem}[\cite{DFCO1}]\label{curv2s2n+1,1}
Let $\gamma:I\to\mathbb{S}^{2n+1}(1)$, $n\geq 2$, be a
proper-biharmonic Legendre curve parametrized by arc length. Then
the parametric equation of $\gamma$ in the Euclidean space
$\mathbb{E}^{2n+2}=(\mathbb{R}^{2n+2},\langle,\rangle)$ is either
$$
\gamma(s)=\frac{1}{\sqrt{2}}\cos\Big(\sqrt{2}s\Big)e_{1}+
\frac{1}{\sqrt{2}}\sin\Big(\sqrt{2}s\Big)e_{2}+\frac{1}{\sqrt{2}}e_{3},
$$
where $\{e_{i},\mathcal{I}e_{j}\}$ are constant unit vectors
orthogonal to each other, or
$$
\gamma(s)=\frac{1}{\sqrt{2}}\cos(As)e_{1}+\frac{1}{\sqrt{2}}\sin(As)e_{2}
+\frac{1}{\sqrt{2}}\cos(Bs)e_{3}+\frac{1}{\sqrt{2}}\sin(Bs)e_{4},
$$
where
$$
A=\sqrt{1+\kappa_{1}}, \quad B=\sqrt{1-\kappa_{1}}, \quad
\kappa_{1}\in(0,1),
$$
$\{e_{i}\}$ are constant unit vectors orthogonal to each other such
that
$$
\langle e_{1},\mathcal{I}e_{3}\rangle=\langle
e_{1},\mathcal{I}e_{4}\rangle=\langle e_{2},
\mathcal{I}e_{3}\rangle=\langle e_{2},\mathcal{I}e_{4}\rangle=0\
$$
$$
A\langle e_{1},\mathcal{I}e_{2}\rangle+B\langle
e_{3},\mathcal{I}e_{4}\rangle=0
$$
and $\mathcal{I}$ is the usual complex structure on
$\mathbb{R}^{2n+2}$.
\end{theorem}

\begin{remark} For the Cases II
and III we also obtained the explicit equations of
proper-biharmonic Legendre curves in odd dimensional spheres
endowed with the deformed Sasakian structure introduced in
\cite{Tanno}.
\end{remark}

In \cite{SMYO} are introduced the complex torsions for a Frenet
curve in a complex manifold. In the same way, for
$\gamma:I\rightarrow N$ a Legendre Frenet curve of osculating
order $r$ in a Sasakian manifold $(N^{2n+1},\varphi,\xi,\eta,g)$,
we define the $\varphi$-torsions $\tau_{ij}=g(E_{i},\varphi
E_{j})=-g(\varphi E_{i},E_{j})$, $i,j=1,\ldots,r$, $i<j$.

It is easy to see that

\begin{proposition}
Let $\gamma:I\to N(c)$ be a proper-biharmonic Legendre Frenet
curve in a Sasakian space form $N(c)$, $c\neq 1$. Then $c>-3$ and
$\tau_{12}$ is constant.
\end{proposition}

Moreover

\begin{proposition}\label{p1} If $\gamma$ is a proper-biharmonic Legendre
Frenet curve in a Sasakian space form $N(c)$, $c>-3$, $c\neq1$, of
osculating order $r<4$, then it is a circle or a helix with
constant $\varphi$-torsions.
\end{proposition}

\begin{proof} From Theorems \ref{teocase2}, \ref{teocase3} and \ref{teocase4} we
see that if $\gamma$ is a proper-biharmonic Legendre Frenet curve
of osculating order $r<4$, then $\tau_{12}=0$ or $\tau_{12}=\pm1$
and, obviously, we only have to prove that when $\gamma$ is a
helix then $\tau_{13}$ and $\tau_{23}$ are constants.

Indeed, by using the Frenet equations of $\gamma$, we have
\begin{eqnarray*}
\tau_{13}&=& g(E_{1},\varphi E_{3})=-\frac{1}{\kappa_{2}}g(\varphi
E_{1},\nabla_{E_{1}}E_{2}+\kappa_{1}E_{1})=-\frac{1}{\kappa_{2}}g(\varphi
E_{1},\nabla_{E_{1}}E_{2})\\ \\
&=&\frac{1}{\kappa_{2}}g(E_{2},\nabla_{E_{1}}\varphi
E_{1})=\frac{1}{\kappa_{2}}g(E_{2},\varphi\nabla_{E_{1}}E_{1}+\xi)=0
\end{eqnarray*}
since
$$
g(E_{2},\xi)=\frac{1}{\kappa_{1}}g(\nabla_{E_{1}}E_{1},\xi)=
-\frac{1}{\kappa_{1}}g(E_{1},\nabla_{E_{1}}\xi)=\frac{1}{\kappa_{1}}g(E_{1},\varphi
E_{1})=0.
$$

On the other hand, it is easy to see that for any Frenet curve of
osculating order $3$ we have
$\tau_{23}=\frac{1}{\kappa_{1}}(\tau_{13}'+\kappa_{2}\tau_{12}+\eta(E_{3}))$
and
\begin{eqnarray*}
\eta(E_{3})&=&g(E_3,\xi)=\frac{1}{\kappa_{2}}\Big(g(\nabla_{E_1}E_{2},\xi)
+\kappa_{1}g(E_1,\xi)\Big)
=-\frac{1}{\kappa_2}g(E_{2},\nabla_{E_{1}}\xi)\\
\\&=&-\frac{1}{\kappa_{2}}\tau_{12}.
\end{eqnarray*}
In conclusion,
$\tau_{23}=\frac{1}{\kappa_{1}}(\tau_{13}'+\kappa_{2}\tau_{12}
-\frac{1}{\kappa_{2}}\tau_{12})=\textnormal{constant}$.

\end{proof}

\begin{proposition} If $\gamma$ is a proper-biharmonic Legendre
Frenet curve in a Sasakian space form $N(c)$ of osculating order
$r=4$, then $c\in(\frac{7}{3},5)$ and the curvatures of $\gamma$ are
$$
\kappa_{1}=\frac{\sqrt{c+3}}{2}, \quad
\kappa_2=\frac{1}{2}\sqrt{\frac{6(c-1)(5-c)}{c+3}}, \quad
\kappa_{3}=\frac{1}{2}\sqrt{\frac{3(c-1)(3c-7)}{c+3}}.
$$
Moreover, the $\varphi$-torsions of $\gamma$ are given by
$$
\left\{\begin{array}{ccc} \tau_{12}=\mp\sqrt{\frac{2(5-c)}{c+3}}, &
\tau_{13}=0, & \tau_{14}=\pm\sqrt{\frac{3c-7}{c+3}}, \\ \\
\tau_{23}=\mp\frac{3c-7}{\sqrt{3(c-1)(c+3)}}, & \tau_{24}=0, &
\tau_{34}=\pm\sqrt{\frac{2(5-c)(3c-7)}{3(c-1)(c+3)}}.
\end{array}\right.
$$
\end{proposition}

\begin{proof}
Let $\gamma$ be a proper-biharmonic Legendre Frenet curve in $N(c)$
of osculating order $r=4$. Then $c\neq 1$ and $\tau_{12}$ is
different from $0$, $1$ or $-1$. From Theorem \ref{teocase4} we have
$\varphi E_{1}=\cos\alpha_{0}E_{2}+\sin\alpha_{0}E_{4}$. It results
that
$$
\tau_{12}=-\cos\alpha_0, \quad \tau_{13}=0, \quad
\tau_{14}=-\sin\alpha_0, \quad \hbox{and} \quad \tau_{24}=0.
$$
In order to prove that $\tau_{23}$ is constant we differentiate
the expression of $\varphi E_{1}$ along $\gamma$ and using the
Frenet equations we obtain
\begin{eqnarray*}
\nabla_{E_1}\varphi E_1&=&\cos\alpha_0\nabla_{E_1}E_2
+\sin\alpha_0\nabla_{E_1}E_4 \\
&=&-\kappa_1\cos\alpha_0
E_1+(\kappa_2\cos\alpha_0-\kappa_3\sin\alpha_0)E_3.
\end{eqnarray*}
On the other hand, $\nabla_{E_1}\varphi E_1=\kappa_1\varphi
E_2+\xi$ and therefore we have
\begin{equation}\label{eq:2}
\kappa_1\varphi E_2+\xi=-\kappa_1\cos\alpha_0
E_1+(\kappa_2\cos\alpha_0-\kappa_3\sin\alpha_0)E_3.
\end{equation}
We take the scalar product in \eqref{eq:2} with $\xi$ and obtain
\begin{equation}\label{eq:2.1}
(\kappa_2\cos\alpha_0-\kappa_3\sin\alpha_0)\eta(E_3)=1.
\end{equation}
In the same way as in the proof of Proposition \ref{p1} we get
\begin{eqnarray*}
\eta(E_{3})&=&g(E_3,\xi)=\frac{1}{\kappa_{2}}\Big(g(\nabla_{E_1}E_{2},\xi)+\kappa_{1}g(E_1,\xi)\Big)
=-\frac{1}{\kappa_2}g(E_{2},\nabla_{E_{1}}\xi)\\
\\&=&-\frac{1}{\kappa_{2}}\tau_{12}=\frac{\cos\alpha_0}{\kappa_2}
\end{eqnarray*}
and then, from \eqref{eq:2.1},
$$
\kappa_2\sin\alpha_0=-\kappa_3\cos\alpha_0.
$$
Therefore $\alpha_0\in(\frac{\pi}{2},\pi)\cup(\frac{3\pi}{2},2\pi)$.

\noindent Next, from Theorem \ref{teocase4}, we have
$$
\kappa_{1}^{2}=\frac{c+3}{4}, \quad
\kappa^{2}_2=\frac{3(c-1)}{4}\cos^{2}\alpha_0, \quad
\kappa^{2}_{3}=\frac{3(c-1)}{4}\sin^{2}\alpha_0,
$$
and so $c$ must be greater than $1$.

Now, we take the scalar product in \eqref{eq:2} with $E_3$, $\varphi
E_2$ and $\varphi E_4$, respectively, and we get
\begin{equation}\label{eq:3}
\kappa_1\tau_{23}=-(\kappa_2\cos\alpha_0-\kappa_3\sin\alpha_0)+\eta(E_{3})=
-\frac{\kappa_2}{\cos\alpha_0}+\frac{\cos\alpha_0}{\kappa_2}
\end{equation}
\begin{equation}\label{eq:4}
\kappa_1\sin^2\alpha_0=-(\kappa_2\cos\alpha_0-\kappa_3\sin\alpha_0)\tau_{23}=
-\frac{\kappa_2}{\cos\alpha_0}\tau_{23}
\end{equation}
\begin{equation}\label{eq:5}
0=\kappa_1\cos\alpha_0\sin\alpha_0+(\kappa_2\cos\alpha_0-\kappa_3\sin\alpha_0)\tau_{34}=
\kappa_1\cos\alpha_0\sin\alpha_0+\frac{\kappa_2}{\cos\alpha_0}\tau_{34}.
\end{equation}
and then, equations \eqref{eq:3} and \eqref{eq:4} lead to
$$
\kappa_{1}^{2}\sin^{2}\alpha_0=\frac{\kappa_{2}^{2}}{\cos^{2}\alpha_0}-1.
$$
We come to the conclusion $\sin^{2}\alpha_0=\frac{3c-7}{c+3}$, so
$c\in(\frac{7}{3},5)$, and then we obtain the expressions of the
curvatures and the $\varphi$-torsions.
\end{proof}

\begin{remark} The proper-biharmonic Legendre curves given by
Theorem \ref{curv2s2n+1,1} (for the case $c=1$) have also constant
$\varphi$-torsions.
\end{remark}

\section{A Method To Obtain Biharmonic Submanifolds in a Sasakian
Space Form}

In \cite{DFCO1} we gave a method to obtain proper-biharmonic
anti-invariant submanifolds in a Sasakian space form from
proper-biharmonic integral submanifolds.

\begin{theorem}[\cite{DFCO1}] \label{teorema1}
Let $(N^{2n+1},\varphi,\xi,\eta,g)$ be a strictly regular Sasakian
space form with constant $\varphi$-sectional curvature $c$ and let
$\textbf{i}:M\rightarrow N$ be an $r$-dimensional integral
submanifold of $N$, $1\leq r\leq n$. Consider
$$
F:\widetilde{M}=I\times M\rightarrow N,\ \ \
F(t,p)=\phi_{t}(p)=\phi_{p}(t)
$$
where $I=\mathbb{S}^{1}$ or $I=\mathbb{R}$ and $\{\phi_{t}\}_{t\in
I}$ is the flow of the vector field $\xi$. Then
$F:(\widetilde{M},\widetilde{g}=dt^{2}+\textbf{i}^{\ast}g)\to N$
is a Riemannian immersion and it is proper-biharmonic if and only
if $M$ is a proper-biharmonic submanifold of $N$.
\end{theorem}

The previous Theorem provides a classification result for
proper-biharmonic surfaces in a Sasakian space form, which are
invariant under the flow-action of $\xi$.

\begin{theorem}[\cite{DFCO1}] Let $M^2$ be a surface of $N^{2n+1}(c)$ invariant under
the flow-action of the characteristic vector field $\xi$. Then $M$
is proper-biharmonic if and only if, locally, it is given by
$x(t,s)=\phi_t(\gamma(s))$, where $\gamma$ is a proper-biharmonic
Legendre curve.
\end{theorem}

Also, using the standard Sasakian $3$-structure on $\mathbb{S}^7$,
by iteration, Theorem \ref{teorema1} leads to examples of
$3$-dimensional proper-biharmonic submanifolds of $\mathbb{S}^7$.

\section{Biharmonic Hopf Cylinders in a Sasakian Space Form}

Let $(N^{2n+1},\varphi,\xi,\eta,g)$ be a strictly regular Sasakian
manifold and $\bar{\bf i}:\bar{M}\rightarrow\bar N$ a submanifold
of $\bar{N}$. Then $M=\pi^{-1}(\bar{M})$ is the Hopf cylinder over
$\bar{M}$, where $\pi:N\rightarrow \bar N=N/\xi$ is the
Boothby-Wang fibration.

In \cite{Ino} the biharmonic Hopf cylinders in a 3-dimensional
Sasakian space form are classified.

\begin{theorem}[\cite{Ino}] Let $S_{\bar\gamma}$ be a Hopf
cylinder, where $\bar\gamma$ is a curve in the orbit space of
$N^{3}(c)$, parametrized by arc length. We have
\begin{enumerate}
\item[1)] if $c\leqslant 1$, then $S_{\bar\gamma}$ is biharmonic
if and only if it is minimal; \item[2)] if $c>1$, then
$S_{\bar\gamma}$ is proper-biharmonic if and only if the curvature
$\bar\kappa$ of $\bar\gamma$ is constant $\bar\kappa^{2}=c-1$.
\end{enumerate}
\end{theorem}

In \cite{DFCO2} we obtained a geometric characterization of
biharmonic Hopf cylinders of any codimension in an arbitrary
Sasakian space form. A special case of our result is the case when
$\bar{M}$ is a hypersurface.

\begin{proposition}[\cite{DFCO2}]\label{t1}
If $\bar{M}$ is a hypersurface of $\bar{N}$, then
$M=\pi^{-1}(\bar{M})$ is biharmonic if and only if
$$
\left\{\begin{array}{ll}\Delta^{\perp}{H}=\Big(-\|B\|^{2}+
\frac{c(n+1)+3n-1}{2}\Big){H}\\ \\2\trace
A_{\nabla^{\perp}_{\cdot}{H}}(\cdot)+n\grad(\|H\|^{2})=0,
\end{array}\right.
$$
where $B$, $A$ and $H$ are the second fundamental form of $M$ in
$N$, the shape operator and the mean curvature vector field,
respectively, and $\nabla^{\perp}$ and $\Delta^{\perp}$ are the
normal connection and Laplacian on the normal bundle of $M$ in
$N$.
\end{proposition}

\begin{proposition}[\cite{DFCO2}]\label{cor1} If $\bar{M}$ is a hypersurface and
$\|\bar{H}\|=\textnormal{constant}\neq 0$, then
$M=\pi^{-1}(\bar{M})$ is proper-biharmonic if and only if
$$
\|\bar{B}\|^{2}=\frac{c(n+1)+3n-5}{2}.
$$
\end{proposition}

\begin{remark} From the last result we see that there exist no
proper-biharmonic hypersurfaces of constant mean curvature
$M=\pi^{-1}(\bar{M})$ in $N(c)$ if $c\leq\frac{5-3n}{n+1}$, which
implies that such hypersurfaces do not exist if $c\leq -3$,
whatever the dimension of $N$ is.
\end{remark}

From now on we shall consider $c>-3$.

\noindent In \cite{Takagi} Takagi classified all homogeneous real
hypersurfaces in the complex projective space $\mathbb{C}P^{n}$,
$n>1$, and found five types of such hypersurfaces (see also
\cite{NiebergallRyan}). The first type (with subtypes $A1$ and $A2$)
are described in the following.

We shall consider $u\in(0,\frac{\pi}{2})$ and $r$ a positive
constant given by $\frac{1}{r^{2}}=\frac{c+3}{4}$.

\begin{theorem}[\cite{Takagi}] The geodesic spheres (Type $A1$) in
complex projective space $\mathbb{C}P^{n}(c+3)$ have two distinct
principal curvatures: $\lambda_{2}=\frac{1}{r}\cot u$ of
multiplicity $2n-2$ and $a=\frac{2}{r}\cot(2u)$ of multiplicity $1$.
\end{theorem}

\begin{theorem}[\cite{Takagi}] The hypersurfaces of Type $A2$ in
complex projective space $\mathbb{C}P^{n}(c+3)$ have three distinct
principal curvatures: $\lambda_{1}=-\frac{1}{r}\tan u$ of
multiplicity $2p$, $\lambda_{2}=\frac{1}{r}\cot u$ of multiplicity
$2q$, and $a=\frac{2}{r}\cot(2u)$ of multiplicity $1$, where $p>0$,
$q>0$, and $p+q=n-1$.
\end{theorem}

We note that if $c=1$ and $\bar{M}$ is of type $A1$ or $A2$ then
$\pi^{-1}(\bar{M})=\mathbb{S}^1(\cos u)\times \mathbb{S}^{2n-1}(\sin
u)\subset\mathbb{S}^{2n+1}$ or
$\pi^{-1}(\bar{M})=\mathbb{S}^{2p+1}(\cos u)\times
\mathbb{S}^{2q+1}(\sin u)$, respectively.

By using Takagi's result we classified in \cite{DFCO2} the
biharmonic Hopf cylinders $M=\pi^{-1}(\bar{M})$ in a Sasakian
space form $N^{2n+1}$ over homogeneous real hypersurfaces in
$\mathbb{C}P^{n}$, $n>1$.

\begin{theorem}[\cite{DFCO2}] Let $M=\pi^{-1}(\bar{M})$ be the Hopf
cylinder over $\bar{M}$.
\begin{enumerate}
\item[1)] If $\bar{M}$ is of Type $A1$, then $M$ is
proper-biharmonic if and only if either
\begin{enumerate}
\item[a)] $c=1$ and $\tan^2 u=1$, or \item[b)]
$c\in\Big[\frac{-3n^{2}+2n+1+8\sqrt{2n-1}}{n^{2}+2n+5},+\infty\Big)\setminus\{1\}$
and
$$
\begin{array}{ll} \tan^2
u=&n+\frac{2c-2}{c+3}\\ \\&\pm\frac{\sqrt{c^{2}(n^{2}+2n+5)
+2c(3n^{2}-2n-1)+9n^{2}-30n+13}}{c+3}.
\end{array}
$$
\end{enumerate}
\item[2)] If $\bar{M}$ is of Type $A2$, then $M$ is
proper-biharmonic if and only if either
\begin{enumerate}
\item[a)] $c=1$, $\tan^2 u=1$ and $p\neq q$, or \item[b)]
$c\in\Big[\frac{-3(p-q)^{2}-4n+4+
8\sqrt{(2p+1)(2q+1)}}{(p-q)^{2}+4n+4},+\infty\Big)\setminus\{1\}$
and
$$
\begin{array}{ll}
\tan^2u=&\frac{n}{2p+1}+\frac{2c-2}{(c+3)(2p+1)}\\
\\&\pm\frac{\sqrt{c^{2}((p-q)^{2}+4n+4)+2c(3(p-q)^{2}+4n-4)
+9(p-q)^{2}-12n+4}}{(c+3)(2p+1)}.
\end{array}
$$

\end{enumerate}
\end{enumerate}
\end{theorem}

\begin{theorem}[\cite{DFCO2}] There are no proper-biharmonic
hypersurfaces $M=\pi^{-1}(\bar{M})$ when $\bar{M}$ is a
hypersurface of Type $B$, $C$, $D$ or $E$ in the complex
projective space $\mathbb{C}P^{n}(c+3)$.
\end{theorem}

\section*{Acknowledgements}

The authors were partially supported by the Grant CEEX, ET,
5871/2006 and by the Grant CEEX, ET, 5883/2006, Romania.

The first author would like to thank to the organizers, especially
to Professor I. Mladenov, for the Conference Grant.

\end{document}